\documentclass[reqno,12pt]{amsart}

\usepackage{amsmath,amssymb,mathrsfs,mathtools,verbatim,bbm} 
\usepackage{xcolor}
\usepackage[backrefs]{amsrefs}
\usepackage[protrusion=true]{microtype}
\usepackage[english]{babel}
\usepackage[margin=1in]{geometry}
\usepackage[onehalfspacing]{setspace}
\usepackage{txfonts}
\usepackage[pdfusetitle,colorlinks,pagebackref,hypertexnames=false,bookmarks=false]{hyperref}
\numberwithin{equation}{section}						
\usepackage[nameinlink,noabbrev]{cleveref}
\expandafter\def\csname ver@etex.sty\endcsname{3000/12/31}

\usepackage{autonum}		
\usepackage[T1]{fontenc}

\let\originalleft\left
\let\originalright\right
\renewcommand{\left}{\mathopen{}\mathclose\bgroup\originalleft}
\renewcommand{\right}{\aftergroup\egroup\originalright}

\makeatletter
\renewcommand*{\eqref}[1]{\hyperref[{#1}]{\textup{\tagform@{\ref*{#1}}}}}		
\makeatother

\newcommand{\rmq}{\rmq}

\newcommand{\rG}{\rm{G}}

\newcommand{\cE}{\mathcal{E}}

\renewcommand{\epsilon}{\varepsilon}

\newcommand{\del}{\partial}

\renewcommand{\Re}{\mathop{\mathrm{Re}}}

\newcommand{\tr}{\mathop{\mathrm{tr}}\nolimits}

\newcommand{\vol}{\mathrm{vol}}

\def\cx{\mathbb{C}}

\def\rl{\mathbb{R}}

\def\dist{\mathrm{dist}}

\def\KW{\scriptscriptstyle \mathrm{KW}}

\def\Re{\mathrm{Re}}
\def\Ric{\mathrm{Ric}}

\def\SU{\mathrm{SU}}
\def\tr{\mathrm{tr}}

\def\Vol{\mathrm{Vol}}

\def\rd{\mathrm{d}}
\def\g{\mathfrak{g}}

\def\<{\mathopen{}\left<}
\def\>{\right>\mathclose{}}
\def\({\mathopen{}\left(}
\def\){\right)\mathclose{}}

\newtheorem{theorem}{Theorem}[section]
\newtheorem{Mtheorem}{Main Theorem}
\newtheorem*{acknowledgment}{Acknowledgment}

\newtheorem{corollary}[theorem]{Corollary}

\newtheorem{definition}[theorem]{Definition}

\newtheorem{lemma}[theorem]{Lemma}

\newtheorem{proposition}[theorem]{Proposition}
\newtheorem{remark}[theorem]{Remark}

\crefname{theorem}{Theorem}{Theorems}						
\creflabelformat{theorem}{#2{#1}#3}							
\crefname{Mtheorem}{Main Theorem}{Main Theorems}			
\creflabelformat{Mtheorem}{#2{#1}#3}						
\crefname{lemma}{Lemma}{Lemmata}							
\creflabelformat{lemma}{#2{#1}#3}							
\crefname{corollary}{Corollary}{Corollaries}				
\creflabelformat{corollary}{#2{#1}#3}						
\crefname{proposition}{Proposition}{Propositions}			
\creflabelformat{proposition}{#2{#1}#3}						
\crefname{ineq}{inequality}{inequalities}					
\creflabelformat{ineq}{#2{\upshape(#1)}#3}					
\crefname{cond}{condition}{conditions}						
\creflabelformat{cond}{#2{\upshape(#1)}#3}					
\crefname{hypoth}{Hypothesis}{Hypotheses}					
\creflabelformat{hypoth}{#2{#1}#3}							
\crefname{def}{Definition}{Definitions}						
\creflabelformat{def}{#2{#1}#3}								
\crefname{appsec}{Appendix}{Appendices}

\title{The Kapustin--Witten equations on ALE and ALF gravitational instantons}

\author{\'Akos Nagy}
\address[\'Akos Nagy]{Duke University, Durham, NC, USA}
\urladdr{\href{https://akosnagy.com}{akosnagy.com}}
\email{\href{mailto:contact@akosnagy.com}{contact@akosnagy.com}}

\author{Gon\c{c}alo Oliveira}
\address[Gon\c{c}alo Oliveira]{Universidade Federal Fluminense IME--GMA, Niter\'oi, Brazil}
\urladdr{\href{https://sites.google.com/view/goncalo-oliveira-math-webpage/home}{sites.google.com/view/goncalo-oliveira-math-webpage/home}}
\email{\href{mailto:galato97@gmail.com}{galato97@gmail.com}}

\date{\today}

\calclayout
\hypersetup{
	pdffitwindow	= true,
	unicode			= true,
	pdffitwindow	= true,
	pdftoolbar		= false,
	pdfmenubar		= false,
	pdfstartview	= FitH,
	pdfkeywords		= {Kapustin--Witten equations, gravitational instantons},
	pdfauthor		= {\'Akos Nagy and Gon\c{c}alo Oliveira},
	colorlinks		= true,
	linkcolor		= black,
	citecolor		= black,
	filecolor		= blue,
	urlcolor		= blue
}

\begin{document}

\begin{abstract}
	We study solutions to the Kapustin--Witten equations on ALE and ALF gravitational instantons. On any such space and for any compact structure group, we prove asymptotic estimates for the Higgs field. We then use it to prove a classification theorem in the case when the underlying manifold is $\rl^4$ or $\rl^3 \times \mathbb{S}^1$ and the structure group is $\SU (2)$.
\end{abstract}

\maketitle


\section{Introduction}

\subsection*{Background}

Let us fix a smooth, oriented, Riemannian 4-manifold, $(M, g)$. Let $\Lambda_\cx^* M = \Lambda^* M \otimes \cx$ be the complexified exterior algebra bundle. Let $\rG$ be a compact Lie group and $P \rightarrow M$ a smooth principal $\rG$-bundle over $M$. Let $\rG_\cx$ be the complex form of $\rG$. We then have a principal $\rG_\cx$-bundle $P_\cx = P \times_{\rG} \rG_\cx$, defined via the adjoint action of $\rG$ on $\rG_\cx$. The Hodge star operator $\ast$ can be extended in two inequivalent ways to $\Lambda_\cx^* M \otimes \g_P \simeq \Lambda^* M \otimes \g_{P_\cx}$, either as a complex linear operator or as a conjugate linear operator. In this paper we consider complexified instantons using the conjugate linear extension. We investigated the complex linear extension in \cite{NO20}.

\smallskip

Let us denote the conjugate linear extension of the Hodge star operator by $\overline{\ast}$. The \emph{Kapustin--Witten equations} can be viewed as complexified self-duality equations as follows: Let $(\nabla, \Phi)$ be a pair consisting of a connection on $P_\cx$ and a section of $\Lambda^1 \otimes \g_{P_\cx}$ (the \emph{Higgs field}). Then $\nabla^\cx \coloneqq \nabla + i \Phi$ is a connection on $P_\cx$. In \cite{KW07}, Kapustin and Witten introduced a family of complexified self-duality equations, parametrized by $\theta \in \rl$, as
\begin{equation}
	\overline{\ast} \left( e^{i \theta} F_{\nabla^\cx} \right) = e^{i \theta} F_{\nabla^\cx}.  \label{eq:thetaKWeq}
\end{equation}
We recommend \cite{GU12} for an introduction to the Kapustin--Witten equations.

Note that when $\Phi = 0$ everywhere and $2 \theta \equiv 0 \: (\textnormal{mod} \: \pi)$, \cref{eq:thetaKWeq} reduces to the classical self-duality equation on $P$. As is standard in complex gauge theory, we break the structure group down from $\rG_\cx$ to $\rG$ by adding the Coulomb type equation $\rd_\nabla^* \Phi = 0$. Since $\g_{P_\cx}$ has a canonical real structure, one can separate the real and imaginary parts of \cref{eq:thetaKWeq} and get the following system of equations:
\begin{subequations}
\begin{align}
	\cos (\theta) \left( F_\nabla - \tfrac{1}{2} [\Phi \wedge \Phi] \right)^- - \sin (\theta) \rd_\nabla^- \Phi	&= 0,  \label{eq:thetaKW_1}  \\
	\sin (\theta) \left( F_\nabla - \tfrac{1}{2} [\Phi \wedge \Phi] \right)^+ + \cos (\theta) \rd_\nabla^+ \Phi	&= 0,  \label{eq:thetaKW_2}  \\
	\rd_\nabla^* \Phi 																							&= 0.  \label{eq:thetaKW_3}
\end{align}
\end{subequations}
The \cref{eq:thetaKW_1,eq:thetaKW_2,eq:thetaKW_3} are called the {\em $\theta$-Kapustin--Witten equations}. When $\theta = \tfrac{\pi}{4}$, \cref{eq:thetaKW_1,eq:thetaKW_2} can be rewritten as the following single equation:
\begin{equation}
	F_\nabla =- \ast \rd_\nabla \Phi + \tfrac{1}{2} [\Phi \wedge \Phi].  \label{eq:theKW}
\end{equation}
There is a Yang--Mills--Higgs type energy functional corresponding to the Kapustin--Witten \cref{eq:theKW}. This functional, which we call the {\em Kapustin--Witten energy} is
\begin{equation}
	\cE_{\KW} (\nabla, \Phi) = \int\limits_M \left( |F_\nabla|^2 + |\nabla \Phi|^2 + \tfrac{1}{4} |[\Phi \wedge \Phi]|^2 \right) \vol.  \label{eq:KWE}
\end{equation}
Similarly to instantons, solutions to the Kapustin--Witten \cref{eq:theKW} with finite Kapustin--Witten energy are, at least formally, absolute minimizers of \eqref{eq:KWE}. When $M$ is closed and $2 \theta \not\equiv 0 \: (\textnormal{mod} \: \pi)$, then all solutions to the $\theta$-Kapustin--Witten \cref{eq:thetaKW_1,eq:thetaKW_2,eq:thetaKW_3} satisfy that $\nabla$ is flat, $\Phi$ is $\nabla$-parallel, and $[\Phi \wedge \Phi]$ vanishes identically; cf. \cite{GU12}*{Corollary 3.3}.

Witten conjectures that the moduli spaces of \cref{eq:theKW} have applications to low dimensional topology; cf. \cites{W12,W18}. Related work has been done recently by, for example, Taubes \cites{T18,T19}, Mazzeo and Witten \cites{MW14,MW17}, He and Mazzeo \cites{HM17}, and He and Walpuski \cite{HW19}.

\smallskip

In this paper, we consider finite energy solutions to the $\theta$-Kapustin--Witten \cref{eq:thetaKW_1,eq:thetaKW_2,eq:thetaKW_3} on certain noncompact, complete, Ricci-flat, Riemannian 4-manifolds, called ALE and ALF gravitational instantons.

Let us, briefly, introduce these classes spaces: Let $(M, g)$ be a (noncompact), smooth, and oriented Riemannian 4-manifold with Levi-Civita connection denoted by $\nabla^{\mathrm{LC}}$. For all $R > 0$ and $x_0 \in M$, let $B_R (x_0) \subset M$ be the (closed) geodesic ball of radius $R$ around $x_0$, and let $S_R (x_0) \coloneqq \partial B_R (x_0)$.

\begin{definition}[ALE and ALF gravitational instantons]
	\label{definition:ALX}
	Let $(M, g)$ be as above and fix $x_0 \in M$. Let $Y$ be a compact 3-manifold, which is a $\mathbb{T}^k$-fibration over a closed base $B$ with projection $\pi_Y$, where $k = 0$, or $1$, together with a connection on $Y$ when $k = 1$, and a metric $g_B$ on $B$. Assume that the end of $X$ is modeled on $Y \times \rl_+$, that is, there exists $R_0 > 0$, such that $B_{R_0} (x_0)$ is a smooth, compact manifold with boundary and $M - B_{R_0} (x_0)$ is diffeomorphic to $Y \times (R_0, \infty)$. Moreover, there exists a diffeomorphism $\phi : M - B_{R_0} (x_0) \rightarrow Y \times (R_0, \infty)$, such that for $j = 0, 1$, and $2$, we have
	\begin{equation}
		\lim\limits_{R \rightarrow \infty} R^j \left\| \left( \nabla^{\mathrm{LC}} \right)^j \left( g - \phi^* \left( dR^2 + g_{\mathbb{T}^k} + R^2 \pi_Y^* (g_B) \right) \right) \right\|_{L^\infty (S_R)} = 0.
	\end{equation}
	We call $(M, g)$ {\em Asymptotically Locally Euclidean (ALE)}, if $k = 0$, and {\em Asymptotically Locally Flat (ALF)}, if $k = 1$. An ALE or ALF 4-manifold is called a {\em gravitational instanton}, if it is Ricci-flat.
\end{definition}

\begin{remark}	
	Note that we do not require $(M, g)$ to be hyperk\"ahler, or even complex. For example, the Euclidean--Schwarzschild manifold can be considered.

	The prototypical example of an ALE gravitational instanton is $\rl^4$ with its canonical flat metric. Other examples are given by the construction of Kronheimer \cite{Kronheimer1989}.

	The prototypical example of an ALF gravitational instanton is $\rl^3 \times \mathbb{S}^1$ with its canonical flat metric. Other important examples include the Euclidean--Schwarzschild, the multi-Taub--NUT, and the Atyah--Hitchin manifolds. Many more (hyperk\"ahler) examples are given via the Gibbons--Hawking construction \cite{GH78}.
\end{remark}

\smallskip

\subsection*{Main results}

Our first main theorem is an asymptotic bound on the Higgs field, $\Phi$, when the underlying manifold, $(M, g)$ is an ALE or ALF gravitational instanton. The proof uses ideas of \cite{JT80}*{Theorem~10.3} adapted to the 4-dimensional setting and to curved geometries.

\begin{Mtheorem}
	\label{Mtheorem:KW_on_ALX}
	Let $(\nabla, \Phi)$ be a finite energy solution to the $\theta$-Kapustin--Witten \cref{eq:thetaKW_1,eq:thetaKW_2,eq:thetaKW_3}, with $2 \theta \not\equiv 0 \: (\textnormal{mod} \: \pi)$, on an ALE or ALF gravitational instanton $(M, g)$. Then there is a constant $c \geqslant 0$, such that
	\begin{equation}
		\lim\limits_{R \rightarrow \infty} \inf\limits_{S_R} |\Phi| = \lim\limits_{R \rightarrow \infty} \sup\limits_{S_R} |\Phi| = \lim\limits_{R \rightarrow \infty} \sup\limits_{M_R} |\Phi| = c.  \label{eq:Phi_limits}
	\end{equation}
	Furthermore, if $c = 0$, then $\Phi = 0$ everywhere.
\end{Mtheorem}

\smallskip

Combining \Cref{Mtheorem:KW_on_ALX} with \cite{T17}*{Theorem~1.1}, we prove the following result.

\begin{corollary}
	\label{Mtheorem:Vanishing_KW}
	Let $(\nabla, \Phi)$ be a finite energy solution to the $\theta$-Kapustin--Witten \cref{eq:thetaKW_1,eq:thetaKW_2,eq:thetaKW_3}, with $2 \theta \not\equiv 0 \: (\textnormal{mod} \: \pi)$, on $M = \rl^4$ or $\rl^3 \times \mathbb{S}^1$ with its flat metric, and let $\rG = \SU (2)$.

Then $[\Phi \wedge \Phi] = 0$. Furthermore, if $M = \rl^4$, then $\nabla$ is flat and $\Phi$ is $\nabla$-parallel, while if $M = \rl^3 \times \mathbb{S}^1$, then either the same conclusion holds or, after passing to a cover of the circle of degree at most two, $(\nabla, \Phi)$ is gauge equivalent to
	\begin{equation}
		\nabla = \nabla' + \cos (2\theta) \phi \, \rd y, \qquad
		\Phi = - \sin (2\theta) \phi \, \rd y,
	\end{equation}
	where $(\nabla', \phi)$ is the pullback of a finite energy monopole $F_{\nabla'} = \ast_3 \rd_{\nabla'} \phi$ on $\rl^3$, and $\rd y$ is the unit parallel 1-form on the circle.\footnote{We are using orientation $\vol_{\rl^3} \wedge \rd y$.}

\end{corollary}

We conjecture that the conclusion $[\Phi \wedge \Phi] = 0$ of \Cref{Mtheorem:Vanishing_KW} holds on an arbitrary ALE or ALF gravitational instantons.

\smallskip

\subsection*{Organization of the paper}  In \Cref{sec:2ndorderKW}, we compute second order equations that are satisfied by solutions to the $\theta$-Kapustin--Witten \cref{eq:thetaKW_1,eq:thetaKW_2,eq:thetaKW_3}, with $2 \theta \not\equiv 0 \: (\textnormal{mod} \: \pi)$. While these equations are known in the literature, we include their proof for clarity and completeness. These are used in the proofs of \Cref{Mtheorem:KW_on_ALX}. In \Cref{sec:e_KW},  we study the analytic properties of the Kapustin--Witten energy density. In \Cref{sec:geometry}, we recall a few useful properties of ALE and ALF gravitational instantons. Finally, in \Cref{sec:proofs} we present the proofs of \Cref{Mtheorem:KW_on_ALX,Mtheorem:Vanishing_KW}.

\smallskip

\begin{acknowledgment}
	The authors are grateful to Mark Stern for many helpful conversations about gauge theory. We also thank Bera Gorapada for his valuable comments on the results of this paper.

	\'Akos Nagy would like to thank Siqi He for useful discussions about the Kapustin--Witten equations, and he also thanks the Universidade Federal Fluminense and IMPA for their hospitality during the final stages of the preparation of this paper.

	Finally, we thank
Juan Ospina for pointing out a missing case in \Cref{Mtheorem:Vanishing_KW}.

	We also thank the anonymous referees for their comments and suggestions.

	Gon\c{c}alo Oliveira is supported by Funda\c{c}\~ao Serrapilheira 1812-27395, by CNPq grants 428959/2018-0 and 307475/2018-2, and FAPERJ through the program Jovem Cientista do Nosso Estado E-26/202.793/2019.
\end{acknowledgment}

\smallskip

\section{The second order Kapustin--Witten equations}
\label{sec:2ndorderKW} 

For the next lemma, let $(x_1, x_2, x_3, x_4)$ be a local, normal chart on $M$ at an arbitrary point, and let
\begin{equation}
	j_\Phi \coloneqq \sum_{i = 1}^4 [\nabla \Phi_i, \Phi_i] \in \Gamma \left( \Lambda^1 \otimes \g_P \right).
\end{equation}
be the {\em supercurrent} generated by $\Phi$.

\begin{lemma}\label{lem:Second_Order_Eqs}
	Let $(M, g)$ be any Riemannian 4-manifold, $P \to M$ a principal $\rG$-bundle, and regard the Ricci tensor of $(M, g)$, $\Ric_g$, as an endomorphism of $\Lambda^1 \otimes \g_P$. If $(\nabla, \Phi)$ is a solution to the $\theta$-Kapustin--Witten \cref{eq:thetaKW_1,eq:thetaKW_2,eq:thetaKW_3}, with $2 \theta \not\equiv 0 \: (\textnormal{mod} \: \pi)$, on $P \to M$, it also satisfies the following system of second order equation:
	\begin{subequations}
	\begin{align}
		\nabla^* \nabla \Phi	&= - \tfrac{1}{2} \ast \left[ \left( \ast [\Phi \wedge \Phi] \right) \wedge \Phi \right] - \Ric_g (\Phi),  \label{eq:Laplace_Phi}  \\		
		\rd_\nabla^* F_\nabla	&= j_\Phi.  \label{eq:d_star_F}
	\end{align}
	\end{subequations}
\end{lemma}

\smallskip

\begin{remark}
	When $(M, g)$ is Ricci-flat, then \cref{eq:Laplace_Phi,eq:d_star_F} are the Euler--Lagrange equations of the Kapustin--Witten energy \eqref{eq:KWE}.
\end{remark}

\smallskip

\begin{remark}
	In \Cref{Mtheorem:KW_on_ALX}, one can replace the condition that $\left( \nabla, \Phi \right)$ is a solution to the $\theta$-Kapustin--Witten \cref{eq:thetaKW_1,eq:thetaKW_2,eq:thetaKW_3} with the assumption that it only solves the second order Kapustin--Witten \cref{eq:Laplace_Phi,eq:d_star_F}, and the conclusions still hold, while in \Cref{Mtheorem:Vanishing_KW}, one can now conclude that $[\Phi \wedge \Phi] = 0$.
\end{remark}

\smallskip

\begin{proof}
	Using the Weitzenb\"ock formula and $\rd_\nabla^* \Phi = 0$, we get
		\begin{equation}
			\nabla^* \nabla \Phi = \rd_\nabla^* \rd_\nabla \Phi - \ast \left[ \left( \ast F_\nabla \right) \wedge \Phi \right] - \Ric_g (\Phi).  \label{eq:Phi_Weitzenbock}
	\end{equation}
	Since $2 \theta \not\equiv 0 \: (\textnormal{mod} \: \pi)$, the number $t = \tan (\theta)$ is defined and nonzero. In this case we may rewrite \cref{eq:thetaKW_1,eq:thetaKW_2} as
	\begin{equation}
		\rd_\nabla^\pm \Phi = \mp t^{\pm 1} \left( F_\nabla - \tfrac{1}{2} [\Phi \wedge \Phi] \right)^\pm.  \label{eq:Intermediate_KW}
	\end{equation}
	Then, writing $\rd_\nabla^{\pm}\Phi=\tfrac{1}{2}(\rd_\nabla \Phi \pm \ast \rd_\nabla \Phi)$, and adding these two equations we find 
	\begin{equation}
		\rd_\nabla \Phi = \tfrac{t^{- 1} - t}{2} \left( F_\nabla - \tfrac{1}{2} [\Phi \wedge \Phi] \right) - \tfrac{t^{- 1} + t}{2} \ast \left( F_\nabla - \tfrac{1}{2} [\Phi \wedge \Phi] \right).  \label{eq:d_Phi_2}  
	\end{equation}
	On the other hand, multiplying \cref{eq:Intermediate_KW} by $t^{\mp 1}$ and adding up the resulting equations yields (after dividing by $t+t^{-1}$)
	\begin{equation}
		\rd_\nabla \Phi = \tfrac{t-t^{- 1}}{t+t^{- 1}} \ast \rd_\nabla \Phi - \tfrac{2}{t+t^{- 1}} \ast \left( F_\nabla - \tfrac{1}{2} [\Phi \wedge \Phi] \right),  \label{eq:d_Phi_1} 
	\end{equation}
	which by rearranging can also be read as
	\begin{equation}
		F_\nabla = \tfrac{1}{2} [\Phi \wedge \Phi] + \tfrac{t - t^{- 1}}{2} \rd_\nabla \Phi - \tfrac{t + t^{- 1}}{2} \ast \rd_\nabla \Phi.  \label{eq:F_nabla}
	\end{equation}
	Thus, using \cref{eq:d_Phi_1,eq:F_nabla}, together with the Bianchi identity $\rd_\nabla^* \ast F_\nabla = 0$, we get
	\begin{subequations}
	\begin{align}
		\rd_\nabla^* \rd_\nabla \Phi	&= \tfrac{t - t^{- 1}}{t+t^{- 1}} \rd_\nabla^* \ast \rd_\nabla \Phi - \tfrac{2}{t + t^{- 1}} \rd_\nabla^* \ast \left( F_\nabla - \tfrac{1}{2} [\Phi \wedge \Phi] \right)  \\
										&= - \tfrac{t - t^{- 1}}{t+t^{- 1}} \ast [F_\nabla \wedge \Phi] - \tfrac{2}{t + t^{- 1} } \ast [\rd_\nabla \Phi \wedge \Phi]  \\
										&= \ast \left[ \ast \left( F_\nabla - \tfrac{1}{2} [\Phi \wedge \Phi] \right) \wedge \Phi \right],
	\end{align}
	\end{subequations}
	where in the last equality we replaced $F_\nabla$ using \cref{eq:F_nabla} and the Jacobi identity $[[\Phi \wedge \Phi] \wedge \Phi] = 0$. Combining the above equation with \cref{eq:Phi_Weitzenbock} concludes the proof of \cref{eq:Laplace_Phi}.

	Now we prove \cref{eq:d_star_F}:
	\begin{align}
		\rd_\nabla^* F_\nabla	&= \tfrac{1}{2} \rd_\nabla^* [\Phi \wedge \Phi] + \tfrac{t - t^{- 1}}{2} \rd_\nabla^* \rd_\nabla \Phi + \tfrac{t+t^{- 1}}{2} \ast [F_\nabla \wedge \Phi]  \\
								&= j_\Phi + \ast [(\ast \rd_\nabla \Phi) \wedge \Phi] + \tfrac{t-t^{- 1}}{2} \ast \left[ \ast \left( F_\nabla - \tfrac{1}{2} [\Phi \wedge \Phi] \right) \wedge \Phi \right] + \tfrac{t+t^{- 1}}{2} \ast [F_\nabla \wedge \Phi]  \\
								&= j_\Phi + \ast [(\ast \rd_\nabla \Phi) \wedge \Phi] + \tfrac{t-t^{- 1}}{2} \ast \left[ \left( \tfrac{t - t^{- 1}}{2}\ast  \rd_\nabla \Phi - \tfrac{t+t^{- 1}}{2} \rd_\nabla \Phi \right) \wedge \Phi \right]  \\
								&\quad+ \tfrac{t+t^{- 1}}{2} \ast \left[ \left( \tfrac{1}{2} [\Phi \wedge \Phi] + \tfrac{t - t^{- 1}}{2} \rd_\nabla \Phi - \tfrac{t+t^{- 1}}{2} \ast \rd_\nabla \Phi \right) \wedge \Phi \right]  \\
								&= j_\Phi + \ast [(\ast \rd_\nabla \Phi) \wedge \Phi] + \tfrac{t-t^{- 1}}{2} \ast \left[ \left( \tfrac{t - t^{- 1}}{2} \ast \rd_\nabla \Phi - \tfrac{t+t^{- 1}}{2} \rd_\nabla \Phi \right) \wedge \Phi \right]  \\
								&\quad+ \tfrac{t+t^{- 1}}{2} \ast \left[ \left( \tfrac{t - t^{- 1}}{2} \rd_\nabla \Phi - \tfrac{t+t^{- 1}}{2} \ast \rd_\nabla \Phi \right) \wedge \Phi \right]  \\
								&= j_\Phi,
	\end{align}
	which completes the proof.
\end{proof}

\smallskip

\section{The Kapustin--Witten energy density}
\label{sec:e_KW}

The Kapustin--Witten energy \eqref{eq:KWE} is the integral of
\begin{equation}
	e_{\KW} = |F_\nabla|^2 + |\nabla \Phi|^2 + \tfrac{1}{4} |[\Phi \wedge \Phi]|^2 \geqslant 0.  \label{eq:KWED}
\end{equation}
We call $e_{\KW}$ the {\em Kapustin--Witten energy density}. First we prove a decay result for $e_{\KW}$.

\begin{proposition}\label{proposition:e_KW_decay}
	Then there is a positive number $C = C (M, g, \rG)$, such that, if $(\nabla, \Phi)$ is a smooth solution to second order \cref{eq:Laplace_Phi,eq:d_star_F} with finite energy, then $e_{\KW}$ decays uniformly to zero at infinity.
\end{proposition}

\begin{proof}
	Let $r_0 = \min \left( \{ \mathrm{inj} (M, g)/2, 1 \} \right)$. Note that $r_0 > 0$, since $(M, g)$ is either ALE or ALF. Let $f = \sqrt{e_{\KW}} \in L_1^2 (B_{2r_0} (x))$. By the Weitzenb\"ock formula and the Ricci-flatness of $(M, g)$, we have (in normal coordinates)
	\begin{equation}
		\nabla^* \nabla (\nabla \Phi) = \nabla (\nabla^* \nabla \Phi) + \sum\limits_{i,j = 1}^4 \left( 2 [F_{ik}, \nabla_k \Phi_j] + [(\rd_\nabla^* F_\nabla)_i, \Phi_j] \right) \ \rd x^i \otimes \rd x^j.
	\end{equation}
	Now further using \cref{eq:Laplace_Phi,eq:d_star_F}, we get that $\Delta |\nabla \Phi|^2 \leqslant C (f^3 + f) - |\nabla^2 \Phi|$. Similar computations for $|F_\nabla|^2$ and $|[\Phi \wedge \Phi]|^2$, and Kato's inequality yields that $f$ weakly satisfies (after maybe redefining $C$) the following inequality:
	\begin{equation}
		\Delta f \leqslant C (f^2 + f)
	\end{equation}
	Hence, by \cite{Uhlenbeck1982}*{Theorem~3.2}, using the notations of the reference
	\begin{equation}
		f = \sqrt{e_{\KW}}, \: b = f + 1, \: q = 2 + \| e_{\KW} \|_{L^\infty (B_{2 \dist (x, x_0)} (x))}^{- 1}, \: n = 4, \: \gamma = 1, \: a_0 = r_0, \mbox{ and } a = \tfrac{1}{2} a_0,
	\end{equation}
	we get that, for some other positive number $C^\prime = C^\prime (M, g)$:
	\begin{equation}
		e_{\KW} (x) \leqslant C^\prime \| e_{\KW} \|_{L^1 (B_{r_0} (x), g)}.  \label[ineq]{ineq:e_KW_bound}
	\end{equation}
	By the finiteness of the energy, we get that the integral of $e_{\KW}$ on $B_{r_0} (x)$ decays uniformly to zero as $\dist (x, x_0) \rightarrow \infty$, and thus so does $e_{\KW} (x)$, which concludes the proof.
\end{proof}

\smallskip

\begin{corollary}
	\label{cor:Lpbounds}
	There is a positive number $C = C (M, g, \rG)$, such that
	\begin{equation}
		\forall p \in [1, \infty) \cup \{ \infty \} : \| e_{\KW} \|_{L^p (M, g)} \leqslant C \| e_{\KW} \|_{L^1 (M, g)}.
	\end{equation}
\end{corollary}

\begin{proof}
	By \cref{ineq:e_KW_bound}, we get that $e_{\KW} \in L^\infty (M)$ and, in fact, $\| e_{\KW} \|_{L^\infty (M)} \leqslant C^\prime \| e_{\KW} \|_{L^1 (M, g)}$.  Hence, for any $p \geqslant 1$, using H\"older's inequality, we get
	\begin{equation}
		\| e_{\KW} \|_{L^p (M, g)} \leqslant \| e_{\KW} \|_{L^\infty (M)}^{(p - 1)/p} \| e_{\KW} \|_{L^1 (M, g)}^{1/p} \leqslant (C^\prime)^{(p - 1)/p} \| e_{\KW} \|_{L^1 (M, g)} \leqslant \max \left( \{ C^\prime, 1 \} \right) \:  \| e_{\KW} \|_{L^1 (M, g)},
	\end{equation}
	thus $\| e_{\KW} \|_{L^p (M, g)}$ is finite, and can be bounded by a constant independent of $p$.
\end{proof}

\smallskip

\section{On the geometry of ALE and ALF gravitational instantons}
\label{sec:geometry}

In this section we recall a few geometric properties of ALE and ALF gravitational instantons that will be used in the proof of \Cref{Mtheorem:KW_on_ALX}.

\smallskip

For quantities $A, B$ the relation $A \lesssim B$ is equivalent to $A = O (B)$, while $A \sim B$ is equivalent to $A \lesssim B$ and $B \lesssim A$. Furthermore, let $k = 0$, if $(M, g)$ is ALE, and $k = 1$, if $(M, g)$ is ALF.

\smallskip

Let $x \in M$, $\rho$ be the radial coordinate on $T_x M$ and define $m \in C^\infty (T_x M; \rl_+)$ via
\begin{equation}
	m \coloneqq \tfrac{\exp_x^* \left( \vol_g \right)}{\rd \rho \wedge \vol_{\mathbb{S}^3}}.
\end{equation}
From the Laplacian Comparison Theorem---see, for example \cite{Walpuski}*{Proposition~20.7}---we have 
\begin{equation}\label{eq:Local_BG}
	\del_\rho ( \rho^{-3} m ) \leqslant 0,  \label[ineq]{ineq:LCT}
\end{equation}
away from the cut locus. Let us now recall the local and global versions of the Gromov--Bishop Theorem, applied to the case of ALE and ALF gravitational instantons.

\begin{theorem}[The Local Bishop--Gromov Theorem for $(M, g)$]
	\label{theorem:Local_Bishop--Gromov}
	As $\rho \to 0$, the quantity $\rho^{-3} m$ converges to a constant, and as $\rho \to \infty$, we have $m \lesssim \rho^{3-k}$.
\end{theorem}

\begin{theorem}[The Global Bishop--Gromov Theorem for $(M, g)$]
	\label{theorem:Bishop--Gromov}
	The quotient, $r^{-4}\Vol(B_r(x))$, is a nonincreasing function of $r$ and converges to a constant as $r \to 0$. As $r \to \infty$, by the \Cref{definition:ALX}, we have $\Vol(B_r(x_0)) \sim r^{4 - k}$.
\end{theorem}

\smallskip

\Cref{theorem:Bishop--Gromov} and \cite{LY86}*{Theorem~5.2} yield the following Lemma.

\begin{lemma}\label{lem:Green_on_M}
	There is a smooth, positive Green's function, $G$ on $(M, g)$, and
	\begin{equation}
		\forall x \in M: \quad \frac{\dist (x, \cdot)^2}{1 + \dist (x, \cdot)^k} \ G (x, \cdot) \in L^\infty \left( M - \{ x \} \right).
	\end{equation}
\end{lemma}

\smallskip

Finally, we present (and prove) the appropriate H\"older--Sobolev Embedding.

\begin{lemma}[The H\"older--Sobolev Embedding Theorem on $(M, g)$]\label{lem:KW_Holder}
	For all $p \in (4, \infty)$, there are positive numbers, $C_{\textnormal{HS}}$ and $R_0$, such that for any $f$ with $|\nabla f| \in L^p (M, g)$ and $x, y \in M$ with $\dist(x,y) \geqslant R_0$, we have
	\begin{equation}
		|f (x) - f (y)| \leqslant C_{\textnormal{HS}} \ \dist (x, y)^{1 - \frac{4 - k}{p}} \ \| \nabla f \|_{L^p (M, g)}.
	\end{equation} 
\end{lemma}

\begin{proof}
	We follow the proof as given in, for  example, \cite{JT80}*{Corollary~2.7}. Let $f$ be as in the statement, $x, y \in M$ arbitrary, let $w$ be the midpoint of a geodesic connecting $x$ and $y$, and $R \coloneqq \tfrac{1}{4} \dist (x, y)$. Then for any $z \in M$, we have $|f (x) - f (y)| \leqslant |f (x) - f (z)| + |f (z) - f (y)|$. If we integrate both sides of this inequality with respect to $z \in B_R(w)$, then we get
	\begin{equation}\label{eq:First_Holder}
		\Vol(B_R(w)) |f (x)-f (y)| = \int\limits_{B_R(w)} \left( |f (x)-f (z)|+ |f (z)-f (y)| \right) \vol(z),
	\end{equation}
	and using $\gamma_{x,z}$ to denote the arc-length parametrized geodesic connecting $x$ to $z$, so $|\dot{\gamma}_{x,z}(t)|=1$, then
	\begin{align}
		\int\limits_{B_R(w)} |f (x)-f (z)| \vol(z)	&\leqslant \int\limits_{B_R(w)} \int\limits_0^{\dist(x,z)} \left| \partial_t f (\gamma_{x,z}(t)) \right| \: \rd t \: \vol(z) \\
													&\leqslant \int\limits_{B_R(w)} \int\limits_0^{\dist(x,z)} |\nabla f (\gamma_{x,z}(t))| |\dot{\gamma}_{x,z}(t)| \: \rd t \: \vol(z)  \\
													&\leqslant \int\limits_0^{3R} \int\limits_{B_R(w)}  |\nabla f (\gamma_{x,z}(t))| \vol(z) \: \rd t,
	\end{align}
	where we have used the triangle inequality to get that $\dist(x,z) \leqslant 3R$. Now we write
	\begin{equation}
		\vol(z) = m (z) \rd \rho \wedge \vol_{\mathbb{S}^3},
	\end{equation}
	and use \cref{ineq:LCT} to deduce that
	\begin{equation}
		m (z) \leqslant \frac{\dist(x,z)^3}{\dist(x, \gamma_{x,z}(t))^3} \ m ( \gamma_{x,z}(t) ) \lesssim R^3 \frac{m ( \gamma_{x,z}(t) ) }{\dist(x, \gamma_{x,z}(t))^3}.
	\end{equation}
	This, together with the fact that $\gamma_{x,z}(t) \in B_{3R}(x)$ for $z \in B_R(w)$ and all $t \in [0,3R]$ yields
	\begin{equation}
		\int\limits_{B_R(w)} |f (x)-f (z)| \vol(z) \lesssim R^4 \int\limits_{B_{3R}(x)}  \frac{|\nabla f (\tilde{z})|}{\dist(x, \tilde{z} )^3} \ \vol(\tilde{z}) \lesssim R^4 \| \dist(x, \cdot )^{-3} \|_{L^q (B_{3 R (x)}, g)} \| \nabla f \|_{L^p (M, g)}.
	\end{equation}
	where we have used H\"older's inequality with conjugate exponents $p$ and $q$. Now, for $\dist(x, \cdot )^{-3}$ to be in $L_{loc}^q (M, g)$ we must have $p > 4$ in which case, for $R \gg 1$ we find, using \Cref{theorem:Local_Bishop--Gromov}, that
	\begin{align}
		\| \dist(x, \cdot )^{-3} \|_{L^q (B_{3R(x)}, g) } &= \left( \: \int\limits_{B_{\epsilon}(x)} \dist(x, \cdot )^{-\frac{3p}{p-1} } \vol + \int\limits_{B_{3R}(x) - B_{\epsilon}(x)} \dist(x, \cdot )^{-\frac{3 p}{p - 1}} \vol \right)^{\frac{p - 1}{p}} \\ 
		&\lesssim \left( 1 + \int\limits_\epsilon^{3R} \rho^{-\frac{3p}{p-1}} \rho^{3-k} \rd \rho \right)^{\frac{p - 1}{p}} \\ 
		&\lesssim R^{(4 - k)\frac{p - 1}{p} - 3} = R^{1 - k - \frac{4 - k}{p}}.
	\end{align}
	Using a similar trick to control the integral of $ |f (z)-f (y)|$ and inserting into \cref{eq:First_Holder} we find that for $R \gg 1$
	\begin{equation}
		|f (x) - f (y)| \lesssim \tfrac{R^4}{\Vol \left( B_R (w) \right)} R^{1 - k - \frac{4 - k}{p}} \ \|\nabla f\|_{L^p (M, g)} \lesssim R^{1 - \frac{4 - k}{p}} \ \|\nabla f\|_{L^p (M, g)} ,
	\end{equation}
	where we have used \cref{theorem:Bishop--Gromov} to bound the volume of the balls.
\end{proof}

\smallskip

\section{The proofs of \Cref{Mtheorem:KW_on_ALX,Mtheorem:Vanishing_KW}}
\label{sec:proofs}

Let $(\nabla, \Phi)$ be a finite energy solution to the second order \cref{eq:Laplace_Phi,eq:d_star_F} on an ALE of ALF gravitational instanton, $(M, g)$.

First of all, $\tfrac{1}{2} |\Phi|^2$ is subharmonic, because
\begin{equation}
	\Delta \left( \tfrac{1}{2} |\Phi|^2 \right) = \Re \left( \langle \Phi , \nabla^* \nabla \Phi \rangle \right) - |\nabla \Phi|^2 = - \tfrac{1}{4} |[\Phi \wedge \Phi]|^2 - |\nabla \Phi|^2 \leqslant 0.
\end{equation}
Let $G$ be the Green's function of $(M, g)$ from \Cref{lem:Green_on_M}. We then define a nonnegative function
\begin{equation}
	w (x) \coloneqq \int\limits_M G (x, y) \left( |\nabla \Phi (y)|^2 + \tfrac{1}{4} |[\Phi (y) \wedge \Phi (y)]|^2 \right) \vol (y). \label{eq:w_def}
\end{equation}

\smallskip

We now can prove \Cref{Mtheorem:KW_on_ALX}.

\begin{proof}[Proof of \Cref{Mtheorem:KW_on_ALX}]
	The proof below is inspired by \cite{JT80}*{Theorem~10.3}.
	
	Let $r \coloneqq \tfrac{1}{2} \dist (x, x_0)$. Using that $|\nabla \Phi|^2 + \tfrac{1}{4} |[\Phi \wedge \Phi]|^2 \leqslant e_{\KW}$, together with \Cref{lem:Green_on_M} and \Cref{cor:Lpbounds}, we have for all $x \in M$:
	\begin{align}
		0 \leqslant w (x)	&\leqslant \int\limits_M G (x, y) e_{\KW} (y) \vol (y) \\
							&= \left( \: \int\limits_{M - B_r (x)} + \int\limits_{B_r (x)} \: \right) G (x, y) e_{\KW} (y) \vol (y)  \\
							&\lesssim r^{- 1} \| e_{\KW} \|_{L^1 (M, g)} + \| e_{\KW} \|_{L^2 (B_r (x), g)}.
	\end{align}
	Thus $w$ is bounded, and furthermore, $w (x)$ converges uniformly to zero as $\dist (x, x_0) \rightarrow \infty$. The same argument shows that the integrand in \cref{eq:w_def} is absolutely convergent, and thus the smoothness of $w$ also follows. Now, by the construction of $w$, the function $h \coloneqq w + \tfrac{1}{2} |\Phi|^2$ is harmonic. Next, we show that $h = o (\dist (\cdot, x_0))$. For each $R > 0$, define
	\begin{equation}
		m (R) \coloneqq \sup\limits_{x \in B_R (x_0)} |\Phi (x)|^2.
	\end{equation}
	Since $|\Phi|^2$ is subharmonic, the supremum is achieved at some point $\tilde{x}$, with $\dist(\tilde{x}, x_0) = R$, that is $m (R) = |\Phi (\tilde{x})|^2$. Furthermore, by Kato's inequality, $|\rd (|\Phi|^2)| \leqslant 2 |\nabla \Phi| |\Phi| \leqslant 2 e_{\KW}^{1/2} |\Phi|$, and thus, using \Cref{lem:KW_Holder} on $B_R(x_0)$ with $R \gg 1$, we get
	\begin{equation}
		\left| |\Phi (\tilde{x})|^2 - |\Phi (x_0)|^2 \right| \lesssim R^{1 - \frac{4 - k}{p}} \| \Phi e_{\KW}^{1/2} \|_{L^p (B_R (x_0), g)} \lesssim R^{1 - \frac{4 - k}{p}} \| e_{\KW} \|_{L^{p/2} (B_R (x_0), g)}^2 \sqrt{m (R)} \lesssim R^{1 - \frac{4 - k}{p}} \sqrt{m (R)}.
	\end{equation}
	Since $k = 0$, or $1$, we can chose, for example, $p = \frac{3(4-k)}{2} > 4$, and find that
	\begin{equation}
		m (R) \leqslant |\Phi (x_0)|^2 + \left| |\Phi (\tilde{x})|^2 - |\Phi (x_0)|^2 \right| \lesssim |\Phi (x_0)|^2 + R^{1/3} \sqrt{m (R)} \leqslant |\Phi (x_0)|^2 + \tfrac{1}{2} R^{2/3} + \tfrac{1}{2} m (R),
	\end{equation}
	and thus (for $R$ large enough) $m (R) = O \left( R^{2/3} \right)$, which shows that $|\Phi|^2$ grows strictly slower than linearly, and thus $h = w + \tfrac{1}{2} |\Phi|^2$ is harmonic and $o (\dist (\cdot, x_0))$. Therefore, it must be constant by the gradient estimate of Cheng and Yau in \cite{Cheng}*{Section~4}, which is nicely summarized in the form we need in \cite{Li1995}*{Lemma~1.5}. Let this constant be $c$. Clearly, $c \geqslant 0$ and $|\Phi|^2 = 2 c - w$. Since $w$ is nonnegative and converges uniformly to zero at infinity, this proves \cref{eq:Phi_limits}. Finally, since $w$ is nonnegative, if $h = c$ is zero, then so is $\Phi$, which completes the proof of \Cref{Mtheorem:KW_on_ALX}.
\end{proof}

\begin{remark}
	In the cases of the gravitational instantons of type ALG and ALH the proof of \Cref{Mtheorem:KW_on_ALX} given above does not work, because those manifolds do not satisfy the conditions of \cite{LY86}*{Theorem~5.2} and hence do not have positive Green's functions.
\end{remark}

\smallskip

Finally, we prove \Cref{Mtheorem:Vanishing_KW}. This is a classification of finite energy Kapustin--Witten fields on $M = \rl^4$ or $\rl^3 \times \mathbb{S}^1$, equipped with their flat metrics. The proof is a combination of \Cref{Mtheorem:KW_on_ALX} and \cite{T17}*{Theorem~1.1}.

\begin{proof}[Proof of \Cref{Mtheorem:Vanishing_KW}]
	Let $M$ be either $\rl^4$ or $\rl^3 \times \mathbb{S}^1$, and equip it with its standard (flat) metric. Let $(\nabla, \Phi)$ be a finite energy solution to the $\theta$-Kapustin--Witten \cref{eq:thetaKW_1,eq:thetaKW_2,eq:thetaKW_3} with structure group $\rG = \SU (2)$.

	If $M = \rl^3 \times \mathbb{S}^1$, then let us pull back $(\nabla, \Phi)$ to $\rl^4$. In both cases, we get a smooth solution to \cref{eq:Laplace_Phi,eq:d_star_F} on $\rl^4$ with bounded $\Phi$. In particular, $|\Phi|$ has bounded average over spheres, and thus by \cite{T17}*{Theorem~1.1} we get that $[\Phi \wedge \Phi] = 0$.

On $M$, \cref{eq:Laplace_Phi} gives $\nabla^* \nabla \Phi = 0$. If $\Phi = 0$, the conclusion follows from \cref{eq:thetaKW_1,eq:thetaKW_2}. Otherwise, let $c > 0$ be the limit in \Cref{Mtheorem:KW_on_ALX}, so that he maximum principle gives $|\Phi| \leqslant c$.

	Then, in a parallel orthonormal frame, we set $Q_{ij} = \langle \Phi_i, \Phi_j \rangle$. Since $\rG = \SU (2)$, the matrix $Q$ is positive semidefinite and has rank at most one. Then, for every constant $\xi \in \rl^4$,
	\begin{equation}
		\Delta \langle Q \xi, \xi \rangle = - 2 |\nabla (\Phi (\xi))|^2 \leqslant 0.
	\end{equation}
	Let $H_R$ be the entrywise harmonic extension of $Q$ from the boundary of $B_R \subset \rl^4$ (respectively $B_R^3 \times \mathbb{S}^1 \subset \mathbb{R}^3 \times \mathbb{S}^1$). The maximum principle gives $0 \leqslant Q \leqslant H_R$ and $\tr H_R \leqslant c^2$. Taking $R_i \nearrow +\infty$, $H_{R_i}$ subsequentially converges to $H$ with a bounded harmonic entries, which are therefore constant by Liouville's theorem. As $Q$ has rank at most one, we have
	\begin{equation}
		\tr Q \leqslant \lambda_{\max} (H) \leqslant \tr H \leqslant c^2.
	\end{equation}
	On the other hand, as $\tr Q = |\Phi|^2 \rightarrow c^2$, the matrix $H$ has rank one and we can pick a constant unit vector $v$ spanning its image. Then $Q \leqslant H$ implies $\Phi = \varphi v^\flat$, for $\varphi = \Phi (v)$.

	The Coulomb gauge fixing gives $\nabla_v \varphi = 0$. Thus $|\varphi|$ is constant along $v$. If the direction of $v$ is infinite, i.e. if $v^\flat \neq dy$ with $y$ the periodic coordinate on $\mathbb R^3 \times \mathbb S^1$, the limit in \Cref{Mtheorem:KW_on_ALX} then gives $|\varphi| = c$ everywhere. Since $\Delta |\varphi|^2 = - 2 |\nabla \varphi|^2$, we obtain $\nabla \Phi = 0$, and \cref{eq:thetaKW_1,eq:thetaKW_2} imply that $\nabla$ is flat.

	When $M = \rl^3 \times \mathbb{S}^1$ and $v^\flat = dy$, we have $\Phi = \varphi d y$ and can define
	\begin{equation}
		\widetilde{\nabla} = \nabla + \cot (2\theta) \Phi, \qquad
		\phi = - \csc (2\theta) \varphi.
	\end{equation}
	Using \cref{eq:F_nabla} and $\nabla_{\partial_y} \varphi = 0$, we get
	\begin{equation}
		\iota_{\partial_y} F_{\widetilde{\nabla}} = 0, \qquad
		\widetilde{\nabla}_{\partial_y} \phi = 0, \qquad
		F_{\widetilde{\nabla}} = \ast_3 \rd_{\widetilde{\nabla}} \phi.
	\end{equation}
	Furthermore, up to a gauge transformation on $\rl^3 \times \rl$, we can assume that $\widetilde{\nabla}$ and $\phi$ are independent of $y$. Then, writing the spatial connection as $\nabla'$ gives the claimed formulas. In fact, this gauge has a transition $h$ around the circle which satisfies $\rd_{\nabla'} h = 0$ and $h \phi h^{-1} = \phi$. Thus, if $h$ was not central, there would be a parallel unit section $\sigma$ such that $\phi = f \sigma$. The monopole equation would give $F_{\nabla'} = (\ast_3 \rd f) \sigma$, which together with the Bianchi identity implies that $\Delta f = 0$ and so $f$ would be constant by Liouville's theorem, giving $\nabla \Phi = 0$. In this case \cref{eq:thetaKW_1,eq:thetaKW_2} again imply flatness. Conversely, substituting any finite energy monopole in the asserted formulas gives a finite energy solution.

\end{proof}

\smallskip

\begin{remark}
	
In the proof of \Cref{Mtheorem:Vanishing_KW}, the conclusion $[\Phi \wedge \Phi] = 0$ follows from \Cref{Mtheorem:KW_on_ALX} and \cite{T17}*{Theorem~1.1}. Generalizations of the latter theorem would provide corresponding generalizations of this conclusion.

\end{remark}

\bigskip

\section*{Conflict of interest}

On behalf of all authors, the corresponding author states that there is no conflict of interest.

\bibliography{references}
\bibliographystyle{abstract}

\end{document}